\documentclass[a4paper,12pt,reqno]{amsart}
\usepackage{amssymb}
\usepackage{ifthen}
 \usepackage[dvips]{graphicx}
\nonstopmode \numberwithin{equation}{section}
\usepackage{amssymb}
\usepackage{ifthen}
\usepackage{subcaption}
\usepackage{graphicx}
\usepackage{cite}
\usepackage{amsmath}
\usepackage[T1]{fontenc} 
\usepackage[utf8]{inputenc}
\usepackage[usenames,dvipsnames]{color}
\usepackage{color}
\usepackage[english]{babel}
\usepackage{fancyhdr}
\usepackage{fancybox}
\usepackage{tikz}

\nonstopmode\numberwithin{equation}{section}
\newtheorem*{thmA}{Theorem A}
\newtheorem*{thmB}{Theorem B}
\newtheorem*{thmC}{Theorem C}
\newtheorem*{thmD}{Theorem D}

\theoremstyle{plain}

\newtheorem{conj}{Conjecture}

\theoremstyle{definition}
\newtheorem{defn}{Definition}[section]

\newtheorem{thm}{Theorem}[section]
\newtheorem{prob}{Problem}[section]
\newtheorem{cor}{Corollary}[section]

\newtheorem{prop}{Proposition}[section]
\newtheorem{rem}{Remark}[section]
\newtheorem{lem}{Lemma}[section]

\newcounter{minutes}
\divide\time by 60
\newcounter{hours}
\multiply\time by 60
\addtocounter{minutes}{-\time}

\newcounter {own}
\def\theown {\thesection       .\arabic{own}}

\newenvironment{pf}[1][]{%
 \vskip 3mm
 \noindent
 \ifthenelse{\equal{#1}{}}%
  {{\slshape Proof. }}%
  {{\slshape #1.} }%
 }%
{\qed\bigskip}

\newcounter{alphabet}

\def\be{\begin{equation}}
\def\ee{\end{equation}}

\newcommand{\bee}{\begin{enumerate}}
\newcommand{\eee}{\end{enumerate}}

\newcommand{\blem}{\begin{lem}}
\newcommand{\elem}{\end{lem}}
\newcommand{\bthm}{\begin{thm}}
\newcommand{\ethm}{\end{thm}}
\newcommand{\bcor}{\begin{cor}}
\newcommand{\ecor}{\end{cor}}
\newcommand{\beg}{\begin{examp}}
\newcommand{\eeg}{\end{examp}}
\newcommand{\begs}{\begin{examples}}
\newcommand{\eegs}{\end{examples}}

\newcommand{\bdefn}{\begin{defn}}
\newcommand{\edefn}{\end{defn}}

\newcommand{\bprob}{\begin{prob}}
\newcommand{\eprob}{\end{prob}}
\newcommand{\bei}{\begin{itemize}}
\newcommand{\eei}{\end{itemize}}

\newcommand{\bcon}{\begin{conj}}
\newcommand{\econ}{\end{conj}}
\newcommand{\bcons}{\begin{conjs}}
\newcommand{\econs}{\end{conjs}}
\newcommand{\bprop}{\begin{prop}}
\newcommand{\eprop}{\end{prop}}
\newcommand{\br}{\begin{rem}}
\newcommand{\er}{\end{rem}}
\newcommand{\brs}{\begin{rems}}
\newcommand{\ers}{\end{rems}}
\newcommand{\bo}{\begin{obser}}
\newcommand{\eo}{\end{obser}}
\newcommand{\bos}{\begin{obsers}}
\newcommand{\eos}{\end{obsers}}
\newcommand{\bpf}{\begin{pf}}
\newcommand{\epf}{\end{pf}}
\newcommand{\ba}{\begin{array}}
\newcommand{\ea}{\end{array}}
\newcommand{\beq}{\begin{eqnarray}}
\newcommand{\beqq}{\begin{eqnarray*}}
\newcommand{\eeq}{\end{eqnarray}}
\newcommand{\eeqq}{\end{eqnarray*}}

\begin{document}

\title{Bohr inequality and Bohr-Rogosinski inequality for $K$-Quasiconformal harmonic mappings}

\author{Molla Basir Ahamed}
\address{Molla Basir Ahamed, Department of Mathematics, Jadavpur University, Kolkata-700032, West Bengal, India.}
\email{mbahamed.math@jadavpuruniversity.in}

\author{Taimur Rahman}
\address{Taimur Rahman, Department of Mathematics, Jadavpur University, Kolkata-700032, West Bengal, India.}
\email{taimurr.math.rs@jadavpuruniversity.in}

\subjclass[{AMS} Subject Classification:]{Primary 30A10, 30C45, 30C62,  30C50}

\def\thefootnote{}
\footnotetext{ {\tiny File:~\jobname.tex,
printed: \number\year-\number\month-\number\day,
          \thehours.\ifnum\theminutes<10{0}\fi\theminutes }
} \makeatletter\def\thefootnote{\@arabic\c@footnote}\makeatother

\begin{abstract}
In this paper, we prove several sharp Bohr-type and Bohr–Rogosinski-type inequalities for $K$-quasiconformal, sense-preserving harmonic mappings on $\mathbb{D}$, whose analytic part is subordinate to a function belonging to the class of concave univalent functions on $\mathbb{D}$. In addition, we derive Bohr-type inequalities for $K$-quasiconformal, sense-preserving harmonic mappings on $\mathbb{D}$, where the analytic part is subordinate to a function from the Ma–Minda class of convex and starlike functions. The results generalize several existing results.
\end{abstract}

\maketitle
\pagestyle{myheadings}
\markboth{M. B. Ahamed and T. Rahman}{Bohr's phenomenon for $K$-quasiconformal harmonic mappings}
\section{Introduction}
The notable work of Bohr in 1914 \cite{Bohr-PLMS-1914} on power series in complex analysis has inspired extensive research in both complex analysis and related fields. This contribution, widely known as Bohr’s phenomenon, has been studied for various function spaces, leading to numerous remarkable developments in recent years. For further information, we refer to the recent expository article by Abu-Muhanna \emph{et al.} \cite{Abu-Muhanna-Ali-Ponnusami-2016}, the survey chapter by Garcia \emph{et al.} \cite[Chapter 8]{Garcia-Mashreghi-Ross-2018}, and the references cited therein.\vspace{2mm}

\subsection{Classical Bohr inequality and its generalizations:}

Let $\mathcal{B}$ be the class of analytic functions on $\mathbb{D}$ such that $|f(z)| \leq 1$ for every $z \in \mathbb{D}$. The classical Bohr inequality, published by Hardy in 1914 following his correspondence with Harald Bohr (see \cite{Bohr-PLMS-1914}), is given by:
\begin{thmA}
	If $f(z)=\sum_{n=0}^{\infty}a_nz^n\in\mathcal{B}$, then
	\begin{align}\label{eq-1.1}
		M_f(r):=\sum_{n=0}^{\infty}|a_n|r^n\le 1\;\; \mbox{for}\;\; |z|=r\le\frac{1}{3}.
	\end{align}
	The radius $1/3$ is best possible.
\end{thmA}
Initially, Bohr introduced the inequality \eqref{eq-1.1} for $|z| = r \leq 1/6$. Subsequently, M. Riesz, I. Schur, and F. Wiener independently improved the inequality \eqref{eq-1.1} for $|z| \leq 1/3$ and established that the constant $1/3$ is the best possible. The constant $1/3$ and the inequality \eqref{eq-1.1} are referred to as the Bohr radius and the Bohr inequality for the class $\mathcal{B}$, respectively. Moreover, for the function $f_a$
\begin{align*}
	f_a(z)=\frac{a-z}{1-az}=a-(1-a^2)\sum_{n=1}^{\infty}a^{n-1}z^n,\; z\in\mathbb{D},
\end{align*}
it follows easily that $M_{f_a}(r)>1$ if, and only if, $r>1/(1+2a)$, which,  for $a\to 1^{-}$ shows that the radius $1/3$ is best possible.\vspace{2mm}

In $2000$, Djakov and Ramanujan \cite{Djakov-JA-2000} investigated the same phenomenon from different point of view. In the case $|a_j|$ is replaced by  $|a_j|^p$, they have studied the powered Bohr radius $r_p$ with two different settings, $p\in (0, 2)$ and $p\in (0, 2]$. A conjecture proposed for powered Bohr radius $r_p$ in \cite{Djakov-JA-2000} has
been settled affirmatively in \cite{Kayumov-Ponnusamy-AASFM-2019}. It is important to note that in the majorant series $M_f(r)$ of $f$, by substituting $|a_0|$ with $|a_0|^2$ in Bohr's inequality, the constant $1/3$ can be replaced by $1/2$. Additionally, when $a_0 = 0$ in Theorem A, the exact Bohr radius improves to $1/\sqrt{2}$ (see \cite{Kayumov-Ponnusamy-CMFT-2017}, \cite[Corollary 2.9]{Paulsen-Popescu-Singh-PLMS-2002}, and the recent paper \cite{Ponnusamy-Wirths-CMFT-2020} for a more general result).  In \cite{Paulsen-Singh-2022}, Paulsen \emph{et al.} presented a simple and intriguing proof of the Bohr inequality. In recent times, there has been an extensive body of research on the classical Bohr inequality and its generalizations. The notion of Bohr radius has been studied by Abu-Muhanna and Ali \cite{Abu-Muhanna-CVEE-2010, Abu-Muhanna-Ali-JMAA-2011}, particularly for analytic functions from 
$\mathbb{D}$ to simply connected domains and to the exterior of the closed unit disk in $\mathbb{C}$. Moreover, the Bohr phenomenon for shifted disks and simply connected domains is explored extensively in \cite{Ahamed-Allu-Halder-AFM-2022,Evdoridis-Ponnusamy-Rasila,Fournier-Ruscheweyh-CRM-2010}. Allu and Halder \cite{Allu-Halder-JMAA-2021}, as well as Bhowmik and Das \cite{Bhowmik-Das-JMAA-2018}, have examined this phenomenon in relation to subordination classes in their respective works. Readers interested in this topic are referred to the papers \cite{Abu-Muhanna-CVEE-2010,Abu-Muhanna-Ali-JMAA-2011,Alkhaleefah-Kayumov-Ponnusamy-PAMS-2019,Evdoridis-Ponnusamy-Rasila,Ismagilov-Kayumov-Ponnusamy-JMAA-2020,Kayumov-Ponnusamy-CMFT-2017,Kayumov-Ponnusamy-AASFM-2019} and their references therein.
\subsection{Bohr-Rogosinski inequality and its generalizations:}
 In addition to the Bohr radius, the Rogosinski radius \cite{Rogosinski-MJ-1923} is another concept defined as follows: Let $f(z)=\sum_{n=0}^{\infty}a_nz^n$ be analytic function in $\mathbb{D}$ with $|f(z)|\leq 1$ in $\mathbb{D}$. For every $N\geq 1$, we have $S_N(z)|\leq 1$ for $|z|\leq 1/2$, where $S_N(z):=\sum_{n=0}^{N-1}a_nz^n$ is the partial sum of $f$, and the radius $1/2 $ is the best possible. Inspired by the Rogosinski radius, Kayumov and Ponnusamy [34] introduced the Bohr-Rogosinski sum 
 $R^{f}_N(z)$, which is defined as follows
 \begin{align*}
 	R^{f}_N(z):=|f(z)|+\sum_{n=N}^{\infty}|a_n||z|^n.
 \end{align*}

It can be easily seen that $|S_N(z)|=|f(z)-\sum_{n=N}^{\infty}a_nz^n|\leq R^{f}_N(z)$. Moreover, the Bohr-Rogosinski sum $R^{f}_N(z)$ has a connection to the classical Bohr sum (Majorant series), where $N=1$  and $f(0)$ is substituted with $f(z)$. For an analytic function $f$ in $\mathbb{D}$ with $|f(z)|\leq 1$ in $\mathbb{D}$,  Kayumov and Ponnusamy \cite{Kayumov-Ponnusamy} introduced the Bohr-Rogosinski radius as the greatest $r\in(0,1)$ such that $|R^{f}_N(z)|\leq 1$ for $|z|\leq r$, and proved the following result.
\begin{thmB}\cite{Kayumov-Ponnusamy}
	Let $f(z)=\sum_{n=0}^{\infty}a_nz^n$ be an analytic function in $\mathbb{D}$ and $|f(z)|\leq 1$. Then, for each $N\in\mathbb{N}$, we have 
	\begin{align}
		|f(z)|+\sum_{n=N}^{\infty}|a_n||z|^n\leq 1
	\end{align}
	for $|z|=r\leq R_N$, where $R_N$ is the positive root of the equation $2(1+r)r^N-(1-r)^2=0$. The radius $R_N$ is the best possible.\vspace{2mm}
\end{thmB}
Recently, in \cite{Hamada-Honda-Kohr-AAMP-2025}, Hamada \emph{et al.} investigated the Bohr–Rogosinski radius for holomorphic mappings on
the unit ball of a complex Banach space with values in a higher dimensional complex Banach space. Also. they obtained the Bohr–Rogosinski radius for a class of subordinations on the unit ball of a complex Banach space. \vspace{2mm}

We recall here the differential subordination technique, an important tool that significantly simplifies the solution of numerous problems in geometric function theory. Suppose $f$ and $g$ are analytic functions in $\mathbb{D}$. We say $f$ is subordinate to $g$, denoted by  $f\prec g$, if there exists a Schwarz function $w\in\mathcal{B}$ such that $w(0)$ and $f(z)=g(w(z))$ for every $z\in\mathbb{D}$. It is well-known that if $g$ is univalent in $\mathbb{D}$, then the subordination $f\prec g$ holds if and only if $f(0)=g(0)$ and  $f(\mathbb{D})\subset g(\mathbb{D})$. Moreover, it is clear that the subordination $f\prec g$ implies $|f^{\prime}(0)|\leq |g^{\prime}(0)|$. For further details and results on subordination classes, we refer to \cite[Chapter 6]{Duren-1983}.\vspace{2mm}

\subsection{Distance formulation of the Bohr inequality:}

The study of the Bohr phenomenon in the context of functions defined by subordination was initiated by Abu-Muhanna \cite{Abu-Muhanna-CVEE-2010}. In this article, we represent the class of all analytic functions $f$ in $\mathbb{D}$ that are subordinate to a given univalent function $g$ using the notation $S(g):=\{f:f\prec g\}$. The family $S(g)$ is said to exhibit the Bohr phenomenon if there exists a radius $r_g\in(0,]$ such that, for any $f(z)=\sum_{n=0}^{\infty}a_nz^n\in S(g)$, it follows that
\begin{align}\label{eq-1.3}
	\sum_{n=1}^{\infty}|a_n|r^n\leq d(g(0),\partial g(\mathbb{D}))\;\; \mbox{for all}\;\; |z|=r\leq r_g,
\end{align}
where $d(g(0),\partial g(\mathbb{D}))$ is the Euclidean distance between $g(0)$ and the boundary of $g(\mathbb{D})$.\vspace{2mm}

The Bohr radius for the family $S(g)$ is the largest value of $r_g$ for which the Bohr phenomenon holds. Similarly, the family $S(g)$ is considered to demonstrate the Bohr-Rogosinski phenomenon if there exists $r_{N,g}\in(0,1]$ such that for all $f(z)=\sum_{n=0}^{\infty}a_nz^n\in S(g)$, the inequality
\begin{align}\label{eq-1.4}
	|f(z)|+\sum_{n=N}^{\infty}|a_n|r^n\leq |g(0)|+d(g(0),\partial g(\mathbb{D}))\;\; \mbox{for all}\;\; |z|=r\leq r_{N, g}.
\end{align}
The largest $r_{N,g}$ is known as the Bohr-Rogosinski radius.\vspace{2mm}

If $g(z)=(a-z)/(1-\bar{a}z)$  with $|a|\leq 1$, it follows that $g(\mathbb{D})=\mathbb{D}$,  $S(g)=\mathcal{B}$, and  $d(g(0),g(\mathbb{D})=1-|g(0)|=1-|a|$. Therefore, using \eqref{eq-1.1}, the inequality \eqref{eq-1.3} holds for $|z|\leq 1/3$, and \eqref{eq-1.4} is true with $r_{N,g}=R_N$  according to Theorem A.
\subsection{Harmonic Mappings:}
Suppose $f=u+iv$ is a complex-valued function in a simply connected domain $\Omega$. The function $f$ is said to be harmonic in $\Omega$ if it satisfies the Laplace equation $\nabla f=4f_{z\bar{z}}=0$, implying that the real and imaginary parts, $u$ and $v$, are harmonic in $\Omega$. Note that any harmonic mapping $f$ can be expressed in the form  $f=h+\bar{g}$, where $h$ and $g$ are analytic functions in $\Omega$. The part $h$ is known as the analytic part and $g$ is called the co-analytic part of $f$, with $\overline{g(z)}$ representing the complex conjugate of $g(z)$. This representation is unique up to an additive constant (see \cite{Duren-2004}). According to the inverse function theorem and Lewy's result \cite{Lewy-BAMS-1936}, a harmonic function  $f$ is locally univalent in $\Omega$ if and only if its Jacobian, given by $J_f(z):=|h^{\prime}(z)|^2-|g^{\prime}(z)|^2$, is non-zero in $\Omega$. A harmonic function $f$ is said to be locally univalent and sense-preserving in $\Omega$ if and only if the Jacobian $J_f(z)$ remains positive in $\Omega$. This is equivalent to $h^{\prime}(z)\neq 0$ in $\Omega$ and the dilatation $\omega_f:=\omega=g^{\prime}/h^{\prime}$ of $f$ satisfying $|\omega_f|<1$ (see \cite{Lewy-BAMS-1936}).
\subsection{K-quasiconformal Harmonic Mappings:}

A harmonic function $f=h+\bar{g}$ is called $K$-quasiconformal harmonic in $\mathbb{D}$ if it is locally univalent, sense-preserving, and satisfies $|\omega_f(z)|\leq k<1$ for all $z\in\mathbb{D}$, where $K=(1+k)/(1-k)\geq 1$ (see \cite{Kalaj-MZ-2008,Martio-AASFAI-1968}). It is evident that $k\to 1$ corresponds to $K\to\infty$. In $2018$, Ponnusamy \emph{et al.}\cite{Kayumov-Ponnusamy-Shakirov-MN-2018} considered the class of all sense-preserving harmonic mappings $f=h+\bar{g}$ of the unit disk $\mathbb{D}$, where $h$ and $g$ are analytic with $g(0) = 0$, and determined the Bohr radius under certain conditions imposed on $h$ and $g$. In their $2019$ study, Liu and Ponnusamy \cite{Liu-Ponnusamy-BMMSS-2019} introduced the Bohr radius for $K$-quasiconformal, sense-preserving harmonic functions $f=h+\bar{g}$ defined in the unit disk $\mathbb{D}$, where $h$ is subordinate to some analytic function $\phi$, and posed two conjectures. In a significant contribution, Liu et al.\cite{Liu-Ponnusami-Wang-2020} investigated the Bohr radius for sense-preserving $K$-quasiregular harmonic mappings in $\mathbb{D}$, where $h(z)-h(0)$ is quasi-subordinate to a given analytic function. They not only established sharper forms of four theorems by Liu and Ponnusamy\cite{Liu-Ponnusamy-BMMSS-2019}, but also settled the two conjectures proposed by those authors. Furthermore, Ahamed and Ahammed\cite{Ahamed-Ahammed-MJM-2024} recently derived an improved version of the Bohr inequality for $K$-quasiregular harmonic mappings in the shifted disk 
\begin{align*}
	\Omega_{\gamma}=\left\{z\in\mathbb{C}: \left|z+\frac{\gamma}{1-\gamma}\right|<\frac{1}{1-\gamma},\; \gamma\in[0, 1)\right\}
\end{align*} 
which includes the unit disk $\mathbb{D}$. They also established sharp Bohr inequalities for a class of $K$-quasiconformal harmonic mappings with bounded analytic parts.\vspace{2mm}

In $2019$, Liu and Ponnusamy \cite{Liu-Ponnusamy-BMMSS-2019} derived results for finding the Bohr radii of harmonic mappings in the unit disk $\mathbb{D}$, where the analytic part is subordinate to a given analytic function.
\begin{thmC}\cite{Liu-Ponnusamy-BMMSS-2019}
	Suppose that $f(z)=h(z)+\overline{g(z)}=\sum_{n=0}^{\infty}a_nz^n+\overline{\sum_{n=1}^{\infty}b_nz^n}$ is a sense-preserving $K$-quasiconformal harmonic mapping in $\mathbb{D}$ and $h\prec\phi$, where $\phi$ is univalent and convex in $\mathbb{D}$. Then
	\begin{align*}
		\sum_{n=1}^{\infty}(|a_n|+|b_n|)r^n\leq d(\phi(0),\partial\phi(\mathbb{D}))\;\; \mbox{for}\;\; r\leq \frac{K+1}{5K+1},
	\end{align*}
	The number $(K+1)/(5K+1)$ is sharp.
\end{thmC}
\begin{thmD}\cite{Liu-Ponnusamy-BMMSS-2019}
	Suppose that $f(z)=h(z)+\overline{g(z)}=\sum_{n=0}^{\infty}a_nz^n+\overline{\sum_{n=1}^{\infty}b_nz^n}$ is a sense-preserving $K$-quasiconformal harmonic mapping in $\mathbb{D}$ and $h\prec\phi$, where $\phi$ is analytic and univalent in $\mathbb{D}$. Then
	\begin{align*}
		\sum_{n=1}^{\infty}(|a_n|+|b_n|)r^n\leq d(\phi(0),\partial\phi(\mathbb{D}))\;\; \mbox{for}\;\; r\leq r_u(k),
	\end{align*}
	where $k=(K-1)/(K+1)$ and $r_u(k)\in(0,1)$ is the positive root of the equation 
	\begin{align*}
		(1-r)^2-4r(1+k\sqrt{1+r})=0.
	\end{align*}
\end{thmD}
\section{\bf{Bohr inequality and Bohr-Rogosinski inequality related to the family of concave univalent functions with pole $p$}}
Let the extended complex plane be denoted as $\widehat{\mathbb{C}} := \mathbb{C} \cup {\infty}$. We define $\widehat{C_p}$ as the set of meromorphic univalent functions $f : \mathbb{D} \to \widehat{\mathbb{C}}$ that satisfy the following conditions:
\begin{enumerate}
	\item [(i)] $f$ is analytic in $\mathbb{D}\setminus\{p\}$ and $\widehat{\mathbb{C}}\setminus f(\mathbb{D})$ is convex domain.\vspace{1.3mm}
	
	\item[(ii)] $f$ has a simple pole at the point $p$.
\end{enumerate}
 By applying a suitable rotation, we may assume, without losing generality, that $0 < p < 1$. Each function in $\widehat{C_p}$ is a concave univalent function with a pole at $p \in (0,1)$, and it has the Taylor series representation $f(z) = \sum_{n=0}^{\infty} a_n z^n$ in the disk $\mathbb{D}_p := \{ z \in \mathbb{D} : |z| < p \}$. Let $C_p:=\{f\in\widehat{C_p}:f(0)=f^{\prime}(0)-1=0\}$.\vspace{2mm}
 
 The representation formula for functions in $C_p$ was established by Wirths \cite{Wirths-SMJ-2006} in $2006$. This formulation enabled Bhowmik \emph{et al.} \cite{Bhowmik-Ponnusamy-Wirths-KMJ-2007} to initiate an analysis based on Laurent series expansion around the pole $p$ in 2007, where they derived certain coefficient bounds for the class $\widehat{C_p}$. The class has since been thoroughly analyzed in works such as \cite{Avkhadiev-Pommerenke-Wirths-MN-2004,Avkhadiev-Wirths-FM-2007,Avkhadiev-Wirths-CVEE-2007,Bhowmik-Das-JMAA-2018,Wirths-SMJ-2003}.\vspace{2mm}
 
 In view  of Theorems C or D, it is natural to raise the following questions with regard to the study of the class $\widehat{C}_p$.
 
 \begin{prob}\label{prob-2.1}
 	Let $f=h+\overline{g}$ is a sense-preserving $K$-quasiconformal harmonic mappings in $\mathbb{D}$ whose analytic part is subordinate to a function in the class $\widehat{C}_p$. Can we obtain the sharp Bohr radius for such mapping $f$?
 \end{prob}
 
 \begin{prob}\label{prob-2.2}
 	Can we establish the Bohr-Rogosinski inequality for $K$-quasiconformal harmonic mappings in $\mathbb{D}$ whose analytic part is subordinate to a function in the class $\widehat{C}_p$?
 \end{prob}
 
 In this section, our aim is to give affirmative answer to these Problems. In this regard, we recall here the following Lemmas which will play a key role to prove our main results.
 
 \begin{lem}\label{lem-2.1}\cite{Gangania-Kumar-MJM-2022}
 	Suppose that $h(z)=\sum_{n=0}^{\infty}a_nz^n$ and $g(z)=\sum_{n=0}^{\infty}b_nz^n$ are two analytic functions in $\mathbb{D}$ and $h\prec g$, then for $N\in\mathbb{N}$, we have
 	\begin{align*}
 		\sum_{n=N}^{\infty}|a_n|r^n\leq\sum_{n=N}^{\infty}|b_n|r^n \;\;\; \mbox{holds for}\;\; |z|=r\leq 1/3.
 	\end{align*}
 \end{lem}
 \begin{lem}\label{lem-2.2}\cite{Alkhaleefah-Kayumov-Ponnusamy-PAMS-2019,Liu-Ponnusami-Wang-2020}
 	Suppose that $h(z)=\sum_{n=0}^{\infty}a_nz^n$ and $g(z)=\sum_{n=0}^{\infty}b_nz^n$ are two analytic functions in $\mathbb{D}$. If $|g^{\prime}(z)|\leq k|h^{\prime}(z)|$ in $\mathbb{D}$ for some $k\in(0,1]$, then
 	\begin{align*}
 		\sum_{n=1}^{\infty}|b_n|r^n\leq k\sum_{n=1}^{\infty}|a_n|r^n \;\; \mbox{for}\;\; |z|=r\leq 1/3.
 	\end{align*}
 \end{lem}
We establish the sharp Bohr inequality for harmonic mappings whose analytic part is subordinate to a function in the class $\widehat{C_p}$, where $p \in (0,1)$. 
\begin{thm}\label{thm-2.1}
Suppose that $f(z)=h(z)+\overline{g(z)}=\sum_{n=0}^{\infty}a_nz^n+\overline{\sum_{n=1}^{\infty}b_nz^n}$ is a sense-preserving $K$-quasiconformal harmonic mapping in $\mathbb{D}$ and $h\prec\phi$, where $\phi\in\widehat{C_p}$ and $p\in(0,1)$. Then
\begin{align*}
	\sum_{n=1}^{\infty}(|a_n|+|b_n|)r^n\leq d(\phi(0),\partial\phi(\mathbb{D}))\;\; \mbox{for}\;\; r\leq r_{k,p},
\end{align*}
where $r_{k,p}\in(0,p)$ is the unique root of the equation 
\begin{align*}
	\left(\frac{2K}{K+1}\right)k_p(r)-\frac{p}{(1+p)^2}=0.
\end{align*}
Each number $r_{k,p}$ is sharp.
\end{thm}
\begin{proof}[\bf Proof of Theorem \ref{thm-2.1}]
	For a given $\phi\in\widehat{C_p}$ with $\phi(z)=\sum_{n=0}^{\infty}c_nz^n$, let the Koebe transform of $\phi$,
	\begin{align*}
		F(z)=\frac{\phi\left(\frac{z+a}{1+\bar{a}z}\right)-\phi(a)}{(1-|a|^2)\phi^{\prime}(a)}
	\end{align*}
	for any $a\in\mathbb{D}\setminus\{p\}$. Since $e^{-it}F(ze^{it})\in C_{|\frac{p-a}{1-\bar{a}p}|}$ for some $t\in\mathbb{R}$, it follows from \cite[Chapter 15, p. 137]{Avkhadiev-Wirths-2009} that 
	\begin{align}\label{eq-2.1}
		d(\phi(0),\partial\phi(\mathbb{D}))\geq(p/(1+p)^2)|\phi^{\prime}(0)|.
	\end{align}
	As $F(z)=(\phi(z)-\phi(0))/\phi^{\prime}(0)$ and using the inequality \eqref{eq-2.1}, we have (see \cite[p. 26]{Jenkins-MMJ-1962})
	\begin{align}\label{eq-2.2}
		|c_n|\leq|\phi^{\prime}(0)|\frac{1-p^{2n}}{(1-p^2)p^{n-1}}\leq,\;\;n\geq 1.
	\end{align}
	As $h\prec\phi$ and $\phi\in\widehat{C_p}$, $p\in(0,1)$, by applying Lemma \ref{lem-2.1} and inequality \eqref{eq-2.2}, we have
	\begin{align}\label{eq-2.3}
		\sum_{n=1}^{\infty}|a_n|r^n\leq\sum_{n=1}^{\infty}|c_n|r^n&\leq|\phi^{\prime}(0)|\sum_{n=1}^{\infty} \frac{1-p^{2n}}{(1-p^2)p^{n-1}}r^n\\&=|\phi^{\prime}(0)|k_p(r)\;\;\mbox{for}\;\; |z|=r\leq 1/3\nonumber.
	\end{align}
	
	Since $f$ is $K$-quasiconformal sense-preserving harmonic mapping 
	on $\mathbb{D}$, by applying Schwarz’s Lemma, we get that the dilatation $\omega = g^{\prime}/h^{\prime}$ is analytic in $\mathbb{D}$ and $|\omega(z)|=|g^{\prime}(z)/h^{\prime}(z)|\leq k$ in $\mathbb{D}$, where $K=(1+k)/(1-k)\geq 1$, $k\in[0, 1)$. Therefore, by Lemma \ref{lem-2.2}, we have
	
	\begin{align}\label{eq-2.4}
		\sum_{n=1}^{\infty}|b_n|r^n\leq k\sum_{n=1}^{\infty}|a_n|r^n\leq k|\phi^{\prime}(0)|k_p(r)\;\;\mbox{for}\;\; |z|=r\leq 1/3.
	\end{align}
	By using \eqref{eq-2.3} and \eqref{eq-2.4}, we have 
	\begin{align}\label{eq-2.5}
		B_1(r):=\sum_{n=1}^{\infty}(|a_n|+|b_n|)r^n&\leq(k+1)|\phi^{\prime}(0)|k_p(r)\\&\leq(k+1)\frac{(1+p)^2}{p}d(\phi(0),\partial\phi(\mathbb{D}))k_p(r)\nonumber.
	\end{align}
	Let
	\begin{align*}
		G_1(r):=(k+1)k_p(r)-\frac{p}{(1+p)^2}.
	\end{align*}
	Then, we see that
	\begin{align*}
		G^{\prime}_1(r)=(k+1)k^{\prime}_p(r)>0\;\; \mbox{for}\;\; r\in(0,p)
	\end{align*}
	This implies that $G$ is strictly increasing function of $r\in(0,p)$. Also, we see that
	\begin{align*}
		G_1(0)=-\frac{p}{(1+p)^2}\;\;\mbox{and}\;\;\lim_{r\to p}G_1(r)=\infty.
	\end{align*}
	  It follows that the equation $G_1(r)=0$ has a unique positive root $r_{k,p}\in(0,p)$. Hence, it follows from \eqref{eq-2.5} that $B_1(r)\leq d(\phi(0),\partial\phi(\mathbb{D}))$ for $r\leq r_{k,p}$, where $r_{k,p}$ is the positive root of the equation $G_1(r)=0$.\vspace{2mm}
	
	For the sharpness of the result, we consider the function $f(z)=h(z)+\overline{g(z)}$ in $\mathbb{D}$ such that 
	\begin{align*}
		h(z)=\phi(z)=k_p(z)\;\; \mbox{and}\;\; g(z)=k\lambda k_p(z),
	\end{align*} 
	where $|\lambda|=1$ and $k=(K-1)/(K+1)$. For the function $k_p$, it is well-known that $\widehat{C}\setminus k_p(\mathbb{D})=[-p/(1-p)^2, -p/(1+p)^2]$ (see \cite[p. 137]{Avkhadiev-Wirths-2009}) and hence we obtain $d(\phi(0),\partial\phi(\mathbb{D}))=p/(1+p)^2$. Thus we have
	\begin{align*}
		\sum_{n=1}^{\infty}\left(\left|\frac{1-p^{2n}}{(1-p^2)p^{n-1}}\right|+\left|\frac{k\lambda(1-p^{2n})}{(1-p^2)p^{n-1}}\right|\right)r^n&=(1+k)\sum_{n=1}^{\infty}\frac{1-p^{2n}}{(1-p^2)p^{n-1}}r^n\\&=(1+k)k_p(r)\\&>\frac{p}{(1+p)^2}=d(\phi(0),\partial\phi(\mathbb{D}))
	\end{align*}
	for $r>r_{k,p}$, where $r_{k,p}$ is the positive root of the equation $G_1(r)=0$. This shows that $r_{k,p}$ is the best possible.
\end{proof}
\begin{rem}
 When we take $K=1$ (that is, $k=0$), we find that 
	\begin{align*}
		r_{0,p}=(1+1/p+p)-(\sqrt{p}+1/\sqrt{p})\sqrt{p+1/p}
	\end{align*}
	is the root in the interval $(0,p)$ of the equation $pr^2-2(1+p+p^2)r+p=0$. This indicates that the result in \cite[Corollary 1]{Bhowmik-Das-JMAA-2018} is a special case of Theorem \ref{thm-2.1}.
\end{rem}
In the following, we establish the Bohr-Rogosinski inequality for harmonic mappings whose analytic part is subordinate to a function in the class $\widehat{C_p}$, where $p \in (0,1)$.
 
\begin{thm}\label{thm-2.2}
	Suppose that $f(z)=h(z)+\overline{g(z)}=\sum_{n=0}^{\infty}a_nz^n+\overline{\sum_{n=1}^{\infty}b_nz^n}$ is a sense-preserving $K$-quasiconformal harmonic mapping in $\mathbb{D}$ and $h\prec\phi$, where $\phi\in\widehat{C_p}$ and $p\in(0,1)$. Then
	\begin{align*}
		|h(z)|+\sum_{n=1}^{\infty}(|a_n|+|b_n|)r^n\leq d(\phi(0),\partial\phi(\mathbb{D}))\;\; \mbox{for}\;\; r\leq r^*_{k,p},
	\end{align*}
	where $r^*_{k,p}\in(0,p)$ is the unique root of the equation 
	\begin{align*}
		\left(\frac{3K+1}{K+1}\right)k_p(r)-\frac{p}{(1+p)^2}=0.
	\end{align*}
	Each number $r^*_{k,p}$ is sharp.
\end{thm}
\begin{proof}[\bf Proof of Theorem \ref{thm-2.2}]
	By employing arguments similar to those in the proof of Theorem \ref{thm-2.1} and considering Lemmas \ref{lem-2.1} and \ref{lem-2.2} and equation \eqref{eq-2.2}, we obtain the inequalities \eqref{eq-2.3} and \eqref{eq-2.4}. Since $h\prec\phi$, where $\phi\in\widehat{C_p}$ and $p\in(0,1)$, we have 
	\begin{align}\label{eq-2.6}
		|h(z)|\leq&|h(0)|+|h(z)-h(0)|\\\nonumber\leq &|\phi(0)|+|\phi(z)-\phi(0)|\\\nonumber=&|\phi(0)|+\sum_{n=1}^{\infty}|c_n|r^n\\\nonumber\leq& |\phi(0)|+\frac{(1+p)^2}{p}d(\phi(0),\partial\phi(\mathbb{D}))k_p(r)
	\end{align}
	for $|z|=r\leq 1/3$. Therefore, by using equation \eqref{eq-2.6}
	\begin{align}\label{eq-2.7}
		B_2(r):=&|h(z)|+\sum_{n=1}^{\infty}(|a_n|+|b_n|)r^n\\\leq&|\phi(0)|+(k+2)\frac{(1+p)^2}{p}d(\phi(0),\partial\phi(\mathbb{D}))k_p(r).\nonumber
	\end{align}
	Let
	\begin{align*}
		G_2(r):=(k+2)k_p(r)-\frac{p}{(1+p)^2}.
	\end{align*}
	Then
	\begin{align*}
		G^{\prime}_2(r)=(k+2)k^{\prime}_p(r)>0\;\; \mbox{for}\;\; r\in(0,p)
	\end{align*}
	This implies that $G_2$ is strictly increasing function of $r\in(0,p)$. Also, 
	\begin{align*}
		G_2(0)=-\frac{p}{(1+p)^2}\;\;\mbox{and}\;\; \lim_{r\to p}G_2(r)=\infty.
	\end{align*}
	 It follows that the equation $G_2(r)=0$ has a unique positive root $r^*_{k,p}\in(0,p)$. Hence, it follows from \eqref{eq-2.7} that $B_2(r)\leq |\phi(0)|+d(\phi(0),\partial\phi(\mathbb{D}))$ for $r\leq r^*_{k,p}$, where $r^*_{k,p}$ is the positive root of the equation $G_2(r)=0$.\vspace{2mm}
	
	For the sharpness of the result, we consider the function $f(z)=h(z)+\overline{g(z)}$ in $\mathbb{D}$ such that 
	\begin{align*}
		h(z)=\phi(z)=k_p(z)\;\; \mbox{and}\;\; g(z)=k\lambda k_p(z),
	\end{align*} 
	where $|\lambda|=1$ and $k=(K-1)/(K+1)$. It is well-known that $d(\phi(0),\partial\phi(\mathbb{D}))=p/(1+p)^2$. Hence, for $z=r$, we have 
	\begin{align*}
		|k_p(r)|&+	\sum_{n=1}^{\infty}\left(\left|\frac{1-p^{2n}}{(1-p^2)p^{n-1}}\right|+\left|\frac{k\lambda(1-p^{2n})}{(1-p^2)p^{n-1}}\right|\right)r^n\\&=|k_p(r)|+(1+k)\sum_{n=1}^{\infty}\frac{1-p^{2n}}{(1-p^2)p^{n-1}}r^n\\&=(2+k)k_p(r)\\&>\frac{p}{(1+p)^2}\\&=d(\phi(0),\partial\phi(\mathbb{D}))
	\end{align*}
	for $r>r^*_{k,p}$, where $r^*_{k,p}$ is the positive root of the equation $G_2(r)=0$. This shows that $r^*_{k,p}$ is the best possible.
\end{proof}
 \section{\bf{Bohr inequality related to the Ma-Minda class of convex and starlike functions}}
 Let $\mathcal{S}$ denote the class of univalent analytic functions $f$ in $\mathbb{D}$, normalized by $f(0)=0$ and $f^{\prime}(0)=1$. Let us recall two important subclasses of the class $\mathcal{S}$, the class of starlike functions and the class of convex functions. The class of starlike functions on $\mathbb{D}$, denoted by $\mathcal{S}^*$, consists of all functions $f\in\mathcal{S}$ such that the image $f(\mathbb{D})$ forms a domain that is starlike with respect to the origin. An important analytic characterization states that $f\in\mathcal{S}^*$ if and only if ${\rm Re}\;(zf^{\prime}(z)/f(z))\geq 0$ for all $z\in\mathbb{D}$. The class of convex functions, denoted by $\mathcal{C}$, consists of all functions $f\in\mathcal{S}$ that map the unit disk $\mathbb{D}$ onto a convex domain. Analytically, A function $f$ belongs to $\mathcal{C}$ if and only if ${\rm Re}\;(1+zf^{\prime\prime}(z)/f^{\prime}(z))\geq 0$ for all $z\in\mathbb{D}$. The inclusion $\mathcal{C}\subsetneq\mathcal{S^*}\subsetneq\mathcal{S}$ is well established. Moreover, according to Alexander’s theorem \cite[p. 43]{Duren-1983}, a normalized analytic function $f$ in $\mathbb{D}$ belongs to $\mathcal{C}$ if, and only if, $zf^{\prime}(z)\in\mathcal{S^*}$. Following the introduction of these classes, several new ones, like strongly convex functions and close-to-convex functions, have been extensively researched. More information on these can be found in \cite{Kaplan-MMJ-1952,Robertson-AM-1936}.\vspace{2mm}
 
 Following the idea of analytic criteria for convexity and starlikeness, Ma and Minda \cite{Ma-Minda-1992} introduced new subclasses of starlike and convex functions on the unit disk $\mathbb{D}$.
 \begin{defn}\label{def-3.1}\cite{Ma-Minda-1992}
 	Let $\phi:\mathbb{D}\to\mathbb{C}$ be an analytic univalent function with $\phi(0)=1, \phi^{\prime}(0)>0, {\rm Re}(\phi(z))>0$, symmetric with respect to the real axis, and starlike with respect	to $1$. Define the classes
 	\begin{align*}
 		\mathcal{C}(\phi):=\left\{f\in\mathcal{S}:1+\frac{zf^{\prime\prime}(z)}{f^{\prime}(z)}\prec\phi(z)\right\}
 	\end{align*}
 	and 
 	\begin{align*}
 		\mathcal{S}^*(\phi):=\left\{f\in\mathcal{S}:\frac{zf^{\prime}(z)}{f(z)}\prec\phi(z)\right\}.
 	\end{align*}
 	A function $\phi$ is termed as a Ma-Minda function if it satisfies the assumption outlined in Definition \ref{def-3.1}. It is significant to mention that for certain choices of $\phi$, the classes $\mathcal{S^*}(\phi)$and $\mathcal{C}(\phi)$ generate several important subclasses of starlike and convex functions, respectively. For example, if we choose $\phi(z)=(1+Az)/(1+Bz)$, the classes $\mathcal{C}(\phi)$ and $\mathcal{S}^*(\phi)$ become the Janowski convex class $\mathcal{C}(A,B)$ and the Janowski starlike class $\mathcal{S}^*(A,B)$, where $-1\leq B<A\leq 1$, respectively. The concept of these classes was introduced by Janowski in \cite{Janowski-APM-1973,Janowski-BAPSSSMAP-1973}. For $\phi(z)=(1+(1-2\alpha)z)/(1-z)$, where $0\leq\alpha<1$, the class $\mathcal{S}^*(\phi)$ becomes the family  $\mathcal{S}^*(\alpha)$ of starlike functions of order $\alpha$, and $\mathcal{C}(\phi)$ becomes the family  $\mathcal{C}(\alpha)$ of convex functions of order $\alpha$. The Bohr phenomenon for these classes was explored in \cite{Allu-Halder-JMAA-2021,Hamada-JMAA-2021}.
 \end{defn}
 \begin{lem}\label{lem-3.1}\cite{Ma-Minda-1992}
 	Let $\phi$ be a Ma-Minda function. Then there exist unique functions $k_{\phi}$ and $h_{\phi}$ in $\mathcal{S}$ such that 
 	\begin{align*}
 		1+\frac{zk^{\prime\prime}_{\phi}(z)}{k^{\prime}(z)}=\frac{zh^{\prime}_{\phi}(z)}{h_{\phi}(z)}=\phi(z)\;\; \mbox{for}\;\;z\in\mathbb{D}.
 	\end{align*}
 
 \end{lem}
 	Since $\phi$ is symmetric about the real axis and maps the real line to itself, all its derivatives at 0 will be real numbers, implying that all the Taylor coefficients of $\phi$ are also real. Therefore, by applying Lemma \ref{lem-3.1}, the Taylor expansion of $k_{\phi}$ and $h_{\phi}$ is of the form 
 \begin{align}\label{eq-3.1}
 	k_{\phi}(z)=z+\sum_{n=2}^{\infty}c_nz^n\;\; \mbox{and}\;\; h_{\phi}(z)=z+\sum_{n=2}^{\infty}d_nz^n,
 \end{align}
 where $c_n$ and $d_n$ are real numbers for all $n$. Now, look at the absolute sum of $k_{\phi}$ and $h_{\phi}$, which we shall denote as follows:
 \begin{align}\label{eq-3.2}
 	\hat{k}_{\phi}(r)=r+\sum_{n=2}^{\infty}|c_n|r^n\;\; \mbox{and}\;\;\hat{h}_{\phi}(r)=r+\sum_{n=2}^{\infty}|d_n|r^n.
 \end{align}
  The growth and distortion theorems for the classes $\mathcal{C}(\phi)$ and $\mathcal{S}^*(\phi)$ are provided in the following Lemmas.
  \begin{lem}\label{lem-3.2}\cite{Ma-Minda-1992}
  	Let $f\in\mathcal{C}(\phi)$. Then $f^{\prime}(z)\prec k^{\prime}_{\phi}(z)$ and 
  	\begin{enumerate}
  		\item [(a)] Growth theorem: $-k_{\phi}(-r)\leq|f(z)|\leq k_{\phi}(r)$,\vspace{1.2mm}
  		
  		\item [(b)] Distortion theorem: $k^{\prime}_{\phi}(-r)\leq|f^{\prime}(z)|\leq k^{\prime}_{\phi}(r)$,
  	\end{enumerate}
  	for $|z|=r\in(0,1)$. Equality holds for some $z\neq 0$ if and only if $f$ is a rotation of $k_{\phi}$.
  \end{lem}
  
  \begin{lem}\label{lem-3.3}\cite{Ma-Minda-1992}
  	Let $f\in\mathcal{S}^*(\phi)$. Then
  	\begin{align*}
  		\frac{f(z)}{z}\prec\frac{h_{\phi}(z)}{z}.
  	\end{align*}
  	Moreover, 
  	\begin{align*}
  		-h_{\phi}(-r)\leq|f(z)|\leq h_{\phi}(r),
  	\end{align*}
  	for $|z|=r\in(0,1)$. Equality holds for some $z\neq 0$ if and only if $f$ is a rotation of $h_{\phi}$.
  \end{lem}
  It should be noted that the distortion theorem does not generally apply to all Ma-Minda functions $\phi$. In \cite{Ma-Minda-1992}, Ma and Minda provided a counterexample and established the distortion theorem for functions in $\mathcal{S}^*(\phi)$ by adding two additional conditions on $\phi$, which are presented in the following lemma.
  \begin{lem}\label{lem-3.4}\cite{Ma-Minda-1992}
  Distortion theorem for the class $\mathcal{S}^*(\phi)$: assume that 
  \begin{align*}
  		\min_{|z|=r}|\phi(z)|=\phi(-r)\;\;\; \mbox{and}\;\;\; \max_{|z|=r}|\phi(z)|=\phi(r).
  \end{align*}
  If $f\in\mathcal{S}^*(\phi)$, then 
  \begin{align*}
  	h^{\prime}_{\phi}(-r)\leq|f^{\prime}(z)|\leq h^{\prime}_{\phi}(r),
  \end{align*}
  for $|z|=r\in(0,1)$. Equality holds for some $z\neq 0$ if, and only if, $f$ is a rotation of $h_{\phi}$.
  	
  \end{lem}
  According to \cite{Ma-Minda-1992}, the functions $-k_{\phi}(-r)$ and $-h_{\phi}(-r)$ are increasing on $(0,1)$ and bounded above by $1$, which ensures the existence of the limits $\lim_{r\to 1}-k_{\phi}(-r)$ and $\lim_{r\to 1}-h_{\phi}(-r)$, denoted by $-k_{\phi}(-1)$ and $-h_{\phi}(-1)$, respectively.\vspace{2mm}
  
  The following lemma plays a key role in deriving the results for the class $\mathcal{C}(\phi)$.
  \begin{lem}\label{lem-3.5}\cite{Arora-Vinayak-CVVE-2025}
  	Let $f\in\mathcal{C}(\phi)$ having the Taylor expansion 
  	\begin{align*}
  		f(z)=z+\sum_{n=2}^{\infty}a_nz^n.
  	\end{align*}
  	Also assume equation \eqref{eq-3.1} gives the Taylor expansion of the function $k_{\phi}$. Then 
  	\begin{align*}
  		\sum_{n=1}^{\infty}|a_n|r^n\leq\sum_{n=1}^{\infty}|c_n|r^n\;\;\mbox{for}\;\; r\leq1/3.
  	\end{align*}
  \end{lem}
  For further improvement of Theorems $C$ or $D$, it is natural to raise the following questions about the study of classes $\mathcal{C}(\phi)$ and $\mathcal{S}^*(\phi)$.
  \begin{prob}\label{prob-3.1}
  	Can we establish the sharp Bohr radius for $K$-quasiconformal harmonic mappings in $\mathbb{D}$ whose analytic part is subordinate to a function in the class $\mathcal{C}(\phi)$?
  \end{prob}
  \begin{prob}\label{prob-3.2}
  Can we establish the sharp Bohr radius for $K$-quasiconformal harmonic mappings in $\mathbb{D}$ whose analytic part is subordinate to a function in the class $\mathcal{S}^*(\phi)$?
  \end{prob}
  We will affirmatively address the Problem \ref{prob-3.1} and Problem \ref{prob-3.2}.
  Here, we obtain the sharp Bohr radius for harmonic mappings in which the analytic part is subordinate to a function from the class $\mathcal{C}(\phi)$.
\begin{thm}\label{thm-3.1}
Suppose that $F(z)=f(z)+\overline{g(z)}=\sum_{n=1}^{\infty}a_nz^n+\overline{\sum_{n=1}^{\infty}b_nz^n}$ is a sense-preserving $K$-quasiconformal harmonic mapping in $\mathbb{D}$ and $f\prec h$, where $h\in\mathcal{C}(\phi)$. Consider 
\begin{align*}
	P_{k,\phi}(r):=\left(\frac{2K}{K+1}\right)\hat{k}_\phi(r)+k_\phi(-1)
\end{align*}
where $\hat{k}_\phi$ has the Taylor expansion $r+\sum_{n=2}^{\infty}|c_n|r^n$. If $\lim_{r\to 1}P_{k,\phi}(r)=0$, then 
\begin{align}\label{eq-3.3}
	\sum_{n=1}^{\infty}(|a_n|+|b_n|)r^n\leq d(h(0),\partial h(\mathbb{D}))
\end{align}
holds for $|z|=r\leq\frac{1}{3}$. Otherwise, the inequality \eqref{eq-3.3} satisfies for $|z|=r\leq\min\{\frac{1}{3},r_{k,\phi}\}$ where $r_{k,\phi}$ is the root in $(0,1)$ of the equation $P_{k,\phi}(r)=0$. In this case, the result is sharp if $r_{k,\phi}\leq 1/3$ and the Taylor coefficients of $k_\phi$  are non-negative.
\end{thm}
\begin{rem}
	If we choose $\phi(z) = \frac{1 + Az}{1+Bz}$ ($-1\leq B<A\leq 1$), the Ma-Minda class of convex functions $\mathcal{C}(\phi)$ reduces to the familiar class  $\mathcal{C}(A,B)$ consisting of the Janowski convex functions. For this particular $\phi$, by using Lemma \ref{lem-3.1}, it can be easily seen that the function $k_{\phi}(z)$ takes the following form
	\begin{align*}
		k_{\phi}(z)=\begin{cases}
			\frac{1}{A}\left[(1+Bz)^{\frac{A}{B}}-1\right],\;\; B\neq 0\;\;\mbox{and}\;\; A\neq0\\\frac{1}{B}\log(1+Bz), \;\;\; \;\;\;\; \;\;\;B\neq 0\;\;\mbox{and}\;\; A=0\\\frac{1}{A}(e^{Az}-1),\;\;\;\;\;\;\;\;\;\;\;\;\;\;\;B=0.
		\end{cases}
	\end{align*}
	and the absolute sum of the terms in the Taylor expansion of $k_{\phi}$ is as follows
	\begin{align*}
		\hat{k}_{\phi}(r)=\begin{cases}
			r+\sum_{n=2}^{\infty}\prod_{m=2}^{n}|A-(m-1)B|r^n,\;\; B\neq 0\;\;\mbox{and}\;\; A\neq 0\\r+\sum_{n=2}^{\infty}\frac{|B|^{n-1}}{n}r^n,\;\;\;\;\;\;\;\;\;\;\;\;\;\;\;\;\;\;\;\;\;\;\;\;\;\;\; B\neq 0\;\;\mbox{and}\;\; A=0\\\frac{1}{A}(e^{Ar}-1),\;\;\;\;\;\;\;\;\;\;\;\;\;\;\;\;\;\;\;\;\;\;\;\;\;\;\;\;\;\;\;\;\;\;\;\;\;\; B=0.
		\end{cases}
	\end{align*}
\end{rem}

If we let $\phi(z) = \frac{1 + Az}{1 + Bz}$ ($-1\leq B<A\leq 1$), then Theorem \eqref{thm-3.1} leads directly to the following Corollary.
\begin{cor}\label{cor-3.1}
	Suppose that $F(z)=f(z)+\overline{g(z)}=\sum_{n=1}^{\infty}a_nz^n+\overline{\sum_{n=1}^{\infty}b_nz^n}$ is a sense-preserving $K$-quasiconformal harmonic mapping in $\mathbb{D}$ and $f\prec h$, where $h\in\mathcal{C}(A,B)$. Consider 
	\begin{align*}
&P_{A,B,k}(r)\\&:=\begin{cases}
			\left(\frac{2K}{K+1}\right)\left[r+\sum_{n=2}^{\infty}\prod_{m=0}^{n-2}|A-(m-1)B|r^n\right]+\frac{1}{A}\left[(1-B)^{\frac{A}{B}}-1\right],\;\; B\neq 0\;\;\mbox{and}\;\;A\neq0\\\left(\frac{2K}{K+1}\right)\left(r+\sum_{n=2}^{\infty}\frac{|B|^{n-1}}{n}r^n\right)+\frac{1}{B}\log(1-B),\;\;\;\;\;\;\;\;\;\;\;\;\;\;\;\;\;\;\;\;\;\;\;\;\;\;\;\;\;\;\;\;\;\; B\neq 0\;\;\mbox{and}\;\;A=0\\\frac{1}{A}\left[\left(\frac{2K}{K+1}\right)(e^{Ar}-1)+(e^{-A}-1)\right],\;\;\;\;\;\;\;\;\;\;\;\;\;\;\;\;\;\;\;\;\;\;\;\;\;\;\;\;\;\;\;\;\;\;\;\;\;\;\;\;\;\;\;\;\;\;\;\;\;\; B=0.
		\end{cases}
	\end{align*}
	If $\lim_{r\to 1}P_{A, B, k}(r)=0$, then 
	\begin{align}\label{Eq-3.4}
		\sum_{n=1}^{\infty}(|a_n|+|b_n|)r^n\leq d(h(0),\partial h(\mathbb{D}))
	\end{align}
	holds for $|z|=r\leq1/3$. Otherwise, the inequality \eqref{Eq-3.4} satisfies for $|z|=r\leq\min\{1/3,R_{A,B,k}\}$ where $R_{A,B,k}$ is the root in $(0,1)$ of the equation $P_{A,B,k}(r)=0$. In this case, the result is sharp if $R_{A,B,k}\leq 1/3$ and the Taylor coefficients of $k_\phi$  are non-negative.
\end{cor}
\begin{rem}
	By setting $\phi(z) = \frac{1 + (1 - 2\alpha)z}{1 - z}$ for $\alpha \in [0, 1)$, the Ma–Minda class of convex functions $\mathcal{C}(\phi)$ coincides with the class $\mathcal{C}(\alpha)$ of convex functions of order $\alpha$. For this $\phi$, the function $k_{\phi}$ is given by 
	\begin{align*}
		k_{\phi}(z)=\begin{cases}
			\frac{1-(1-z)^{2\alpha-1}}{2\alpha-1},\;\;\;\;\alpha\neq\frac{1}{2}\\-\log(1-z),\;\;\alpha=\frac{1}{2}.
		\end{cases}
	\end{align*}
	and hence
	\begin{align*}
		\hat{k}_{\phi}(r)=\begin{cases}
			\frac{1-(1-r)^{2\alpha-1}}{2\alpha-1},\;\;\;\;\alpha\neq\frac{1}{2}\\-\log(1-r),\;\;\alpha=\frac{1}{2}.
		\end{cases}
	\end{align*}
\end{rem}
Thus, setting $\phi(z) = \frac{1 + (1-2\alpha)z}{1-z}$ in Theorem \ref{thm-3.1} gives rise to the following corollary.
  
\begin{cor}\label{cor-3.2}
	Suppose that $F(z)=f(z)+\overline{g(z)}=\sum_{n=1}^{\infty}a_nz^n+\overline{\sum_{n=1}^{\infty}b_nz^n}$ is a sense-preserving $K$-quasiconformal harmonic mapping in $\mathbb{D}$ and $f\prec h$, where $h\in\mathcal{C}(\alpha)$. Then 
	\begin{align*}
		\sum_{n=1}^{\infty}(|a_n|+|b_n|)r^n\leq d(h(0),\partial h(\mathbb{D}))
	\end{align*}
	holds for $|z|=r\leq\min\{1/3, R_{k,\alpha}\}$, where $R_{k,\alpha}$ is the root in $(0,1)$ of the equations
	\begin{align*}
		\left(\frac{2K}{K+1}\right)\left(\frac{1-(1-r)^{2\alpha-1}}{2\alpha-1}\right)+\frac{1-2^{2\alpha-1}}{2\alpha-1}=&0,\;\;\mbox{if}\;\;\alpha\neq\frac{1}{2}\\\left(\frac{2K}{K+1}\right)\log(1-r)+\log2=&0,\;\;\mbox{if}\;\;\alpha=\frac{1}{2}
	\end{align*}
	The number $R_{k,\alpha}$ is sharp if $R_{k,\alpha}\leq 1/3$. 
	
\end{cor}
\begin{rem}
	If $0\leq\alpha<1$, $K=1$, and $f=h$, then the Corollary \ref{cor-3.2} coincides with the result in \cite[Theorem 2.2]{Ali-Jain-Ravichandran-RM-2019}.
\end{rem}
\begin{rem}
	Taking $\phi(z) = \frac{1 + z}{1 - z}$, the class $\mathcal{C}(\phi)$ becomes identical to the classical convex class $\mathcal{C}$. With this $\phi$, it is straightforward to obtain $k_{\phi}$ as follows 
	\begin{align*}
		k_{\phi}(z)=\frac{z}{(1-z)^2}\;\;\; \mbox{and}\;\;\; \hat{k}_{\phi}(r)=\frac{r}{(1-r)^2}.
	\end{align*}
\end{rem}
For the choice $\phi(z) = \frac{1 + z}{1 - z}$, Theorem \ref{lem-3.1} immediately gives the result stated in the following corollary.
\begin{cor}\label{cor-3.3}
	
		Suppose that $F(z)=f(z)+\overline{g(z)}=\sum_{n=1}^{\infty}a_nz^n+\overline{\sum_{n=1}^{\infty}b_nz^n}$ is a sense-preserving $K$-quasiconformal harmonic mapping in $\mathbb{D}$ and $f\prec h$, where $h\in\mathcal{C}$. Then
	\begin{align*}
		\sum_{n=1}^{\infty}(|a_n|+|b_n|)r^n\leq d(\phi(0),\partial\phi(\mathbb{D}))\;\; \mbox{for}\;\; r\leq \frac{K+1}{5K+1},
	\end{align*}
	The number $(K+1)/(5K+1)$ is sharp.
\end{cor}
\begin{rem}
	Theorem C provides the Bohr radius for sense-preserving $K$-quasiconformal harmonic mappings in $\mathbb{D}$ with analytic parts subordinate to a convex function, while Corollary \ref{cor-3.3} determines the same Bohr radius for such mappings when the analytic part is subordinate to a normalized convex function.
\end{rem}
\begin{proof}[\bf Proof of Theorem \ref{thm-3.1}]
	Let $h\in\mathcal{C}(\phi)$ with the Taylor expansion $h(z)=z+\sum_{n=2}^{\infty}e_nz^n$. Since $f\prec h$, by applying Lemma \ref{lem-2.1} and Lemma \ref{lem-3.5}, we have
	\begin{align}\label{eq-3.4}
		\sum_{n=1}^{\infty}|a_n|r^n\leq\sum_{n=1}^{\infty}|e_n|r^n\leq\sum_{n=1}^{\infty}|c_n|r^n\;\;\mbox{for}\;\; |z|=r\leq 1/3.
	\end{align}
	By applying the similar arguments as in the proof of Theorem \ref{thm-2.1} and Lemma \ref{lem-2.2}, we obtain
	
	\begin{align}\label{eq-3.5}
		\sum_{n=1}^{\infty}|b_n|r^n\leq k\sum_{n=1}^{\infty}|a_n|r^n\leq k\sum_{n=1}^{\infty}|c_n|r^n \;\;\mbox{for}\;\; |z|=r\leq 1/3.
	\end{align}
	By using \eqref{eq-3.4} and \eqref{eq-3.5}, we have 
	\begin{align}\label{eq-3.6}
		C_1(r):=\sum_{n=1}^{\infty}(|a_n|+|b_n|)r^n\leq& (k+1)\sum_{n=1}^{\infty}|c_n|r^n\\\nonumber=&(k+1)\hat{k}_p(r)\\=&P_{k,\phi}(r)-k_{\phi}(-1)\nonumber,
	\end{align}
	where $P_{k,\phi}(r)=(k+1)\hat{k}_p(r)+k_{\phi}(-1)$. Clearly, $P_{k,\phi}$ is a continuous and strictly increasing function on the interval $(0,1)$. Also, we can deduce the following inequality by letting $r\to 1$ in Lemma \ref{lem-3.2}(a) that
	\begin{align*}
		-k_{\phi}(-1)\leq k_{\phi}(1)\leq\hat{k}_{\phi}(1)\leq(k+1)\hat{k}_{\phi}(1).
	\end{align*}
	It follows that $(k+1)\hat{k}_{\phi}(1)+k_{\phi}(-1)\geq 0$ and hence $\lim_{r\to 1}P_{k,\phi}(r)=(k+1)\hat{k}_{\phi}(1)+k_{\phi}(-1)\geq 0$.\vspace{2mm}
	
	\noindent{\bf Case I:} $\lim_{r\to 1}P_{k,\phi}(r)=0$.\vspace{1.2mm}
	
	Since $P_{k,\phi}(r)$ is strictly increasing function on $(0,1)$, we obtain $P_{k,\phi}(r)<0$ for all $r\in(0,1)$. Also, by letting $r\to 1$ in the inequality of Lemma \ref{lem-3.2}(a), we obtain $-k_{\phi}(-1)\leq d(h(0),\partial h(\mathbb{D}))$. Then from the inequality \eqref{eq-3.6}, we have $C_1(r)\leq-k_{\phi}(-1)\leq d(h(0),\partial h(\mathbb{D}))$ for $|z|=r\leq 1/3$. \vspace{1.2mm} 
	
	\noindent{\bf Case II:} $\lim_{r\to 1}P_{k,\phi}(r)>0$.\vspace{1.2mm}
	
	It is noted that $P_{k,\phi}(0)=k_{\phi}(-1)<0$ ensures the existence of a unique root of $P_{k,\phi}(r)=0$ in $(0,1)$, denoted by $r_{k,\phi}$. Then we have $P_{k,\phi}(r)\leq0$ for $|z|=r\leq r_{k,\phi}$. This gives 
	\begin{align*}
		C_1(r)\leq-k_{\phi}(-1)\leq d(h(0),\partial h(\mathbb{D}))\;\; \mbox{for}\;\; |z|=r\leq \min\{1/3,r_{k,\phi}\}.
	\end{align*}
	To establish the sharpness of the radius obtained in {\bf Case II} when $r_{k,\phi}\leq 1/3$, we consider the function $F(z)=f(z)+\overline{g(z)}$ in $\mathbb{D}$, where $f=h=k_{\phi}$ and $g=k\lambda k_{\phi}$ with non-negative Taylor coefficient of $k_{\phi}$, and $|\lambda|=1$. Then the facts $\hat{k}_{\phi}=k_{\phi}$ and $d(k_{\phi}(0),\partial k_{\phi}(\mathbb{D})=-k_{\phi}(-1)$ give that
	\begin{align*}
		\sum_{n=1}^{\infty}(|c_n|+|k\lambda c_n|)r^n=(k+1)k_{\phi}(r)>-k_{\phi}(-1)=d(k_{\phi}(0),\partial k_{\phi}(\mathbb{D})
	\end{align*}
	holds for $r>r_{k,\phi}$. This shows that the radius $r_{k,\phi}$ cannot be improved. 
\end{proof}
Our focus here is on deriving the sharp Bohr radius for harmonic mappings with an analytic component subordinate to a function in $\mathcal{S}^*(\phi)$.

\begin{thm}\label{thm-3.2}
	Suppose that $F(z)=f(z)+\overline{g(z)}=\sum_{n=1}^{\infty}a_nz^n+\overline{\sum_{n=1}^{\infty}b_nz^n}$ is a sense-preserving $K$-quasiconformal harmonic mapping in $\mathbb{D}$ and $f\prec h$, where $h\in \mathcal{S}^*(\phi)$. Consider 
	\begin{align*}
		Q_{k,\phi}(r):=\left(\frac{2K}{K+1}\right)\hat{h}_\phi(r)+h_\phi(-1)
	\end{align*}
	where $\hat{h}_\phi$ has the Taylor expansion $r+\sum_{n=2}^{\infty}|d_n|r^n$. If $\lim_{r\to 1}Q_{k,\phi}(r)=0$, then 
	\begin{align}\label{eq-3.7}
		\sum_{n=1}^{\infty}(|a_n|+|b_n|)r^n\leq d(h(0),\partial h(\mathbb{D}))
	\end{align}
	holds for $|z|=r\leq1/3$. Otherwise, the inequality \eqref{eq-3.7} satisfies for $|z|=r\leq\min\{1/3,r^*_{k,\phi}\}$ where $r^*_{k,\phi}$ is the root in $(0,1)$ of the equation $Q_{k,\phi}(r)=0$. In this case, the result is sharp if $r^*_{k,\phi}\leq 1/3$ and the Taylor coefficients of $h_\phi$  are non-negative.
\end{thm}

\begin{rem}
	If we choose $\phi(z) = \frac{1 + Az}{1+Bz}$ ($-1\leq B<A\leq 1$), the class $\mathcal{S}^*(\phi)$ reduces to the familiar class  $\mathcal{S}^*(A,B)$ consisting of the Janowski starlike functions. For this particular $\phi$, by using Lemma \ref{lem-3.1}, we obtain that the function  $h_{\phi}(z)$ takes the following form
	\begin{align*}
		h_{\phi}(z)=\begin{cases}
			z(1+Bz)^{\frac{A-B}{B}},\;\; B\neq 0\\ze^{Az},\;\;\;\;\;\;\;\;\;\;\;\;\;\;\;\; B=0.
		\end{cases}
	\end{align*}
	and the absolute corresponding sum of Taylor expansion of $h_{\phi}$ is given by 
	\begin{align*}
		\hat{h}_{\phi}(r)=\begin{cases}
			r+\sum_{n=2}^{\infty}\prod_{m=0}^{n-2}\frac{|(B-A)+Bm|}{m+1}r^n,\;\; B\neq 0\\re^{Ar},\;\;\;\;\;\;\;\;\;\;\;\;\;\;\;\;\;\;\;\;\;\;\;\;\;\;\;\;\;\;\;\;\;\;\;\;\;\;\; B=0.
		\end{cases}
	\end{align*}
\end{rem}
In the special case where the function $\phi(z)$ is defined by $\phi(z) = \frac{1 + Az}{1+Bz}$, Theorem \ref{lem-3.2} does not hold in its most general form, but instead leads to the more specific result outlined in the following corollary. 
\begin{cor}\label{Cor-3.1}
	Suppose that $F(z)=f(z)+\overline{g(z)}=\sum_{n=1}^{\infty}a_nz^n+\overline{\sum_{n=1}^{\infty}b_nz^n}$ is a sense-preserving $K$-quasiconformal harmonic mapping in $\mathbb{D}$ and $f\prec h$, where $h\in\mathcal{S}^*(A,B)$. Consider 
	\begin{align*}
		Q_{A,B,k}(r):=\begin{cases}
			\left(\frac{2K}{K+1}\right)\left(r+\sum_{n=2}^{\infty}\prod_{m=0}^{n-2}\frac{|(B-A)+Bm|}{m+1}r^n\right)-(1-B)^{(A-B)/B},\;\; B\neq 0\\\left(\frac{2K}{K+1}\right)re^{Ar}-e^{-A},\;\;\;\;\;\;\;\;\;\;\;\;\;\;\;\;\;\;\;\;\;\;\;\;\;\;\;\;\;\;\;\;\;\;\;\;\;\;\;\;\;\;\;\;\;\;\;\;\;\;\;\;\;\;\;\;\;\;\;\;\;\; B=0.
		\end{cases}
	\end{align*}
	If $\lim_{r\to 1}Q_{A, B, k}(r)=0$, then 
	\begin{align}\label{Eq-3.8}
		\sum_{n=1}^{\infty}(|a_n|+|b_n|)r^n\leq d(h(0),\partial h(\mathbb{D}))
	\end{align}
	holds for $|z|=r\leq1/3$. Otherwise, the inequality \eqref{Eq-3.8} satisfies for $|z|=r\leq\min\{1/3,R^*_{A,B,k}\}$ where $R^*_{A,B,k}$ is the root in $(0,1)$ of the equation $Q_{A,B,k}(r)=0$. In this case, the result is sharp if $R^*_{A,B,k}\leq 1/3$ and the Taylor coefficients of $h_\phi$  are non-negative.
\end{cor}
\begin{rem}
	If we set $K=1$ (so that $k=0$) and $f=h$, Corollary \ref{Cor-3.1} reduces to the result established in \cite[Theorem 1]{Anand-Jain-Kumar-BMMSS-2021}.
\end{rem}
\begin{rem}
	On taking $\phi(z)=\frac{1+(1-2\alpha)z}{1-z} (0\leq\alpha<1)$, the class $\mathcal{S}^*(\phi)$ becomes the family $\mathcal{S}^*(\alpha)$ of starlike functions of order $\alpha$. With this $\phi$, the function $h_{\phi}$ is given by 
	\begin{align*}
		h_{\phi}(z)=\frac{z}{(1-z)^{2(1-\alpha)}}
	\end{align*}
	and hence
	\begin{align*}
		\hat{h}_{\phi}(r)=\frac{r}{(1-r)^{2(1-\alpha)}}.
	\end{align*}
\end{rem}
For the particular choice $\phi(z) = \frac{1 + (1-2\alpha)z}{1-z}$, Theorem \ref{lem-3.2} leads to the result outlined in the following corollary.  
\begin{cor}\label{Cor-3.2}
	Suppose that $F(z)=f(z)+\overline{g(z)}=\sum_{n=1}^{\infty}a_nz^n+\overline{\sum_{n=1}^{\infty}b_nz^n}$ is a sense-preserving $K$-quasiconformal harmonic mapping in $\mathbb{D}$ and $f\prec h$, where $h\in\mathcal{S}^*(\alpha)$. Then 
	\begin{align*}
		\sum_{n=1}^{\infty}(|a_n|+|b_n|)r^n\leq d(h(0),\partial h(\mathbb{D}))
	\end{align*}
	holds for $|z|=r\leq\min\{1/3, R^*_{k,\alpha}\}$, where $R^*_{k,\alpha}$ is the root in $(0,1)$ of the equation
	\begin{align*}
		\left(\frac{2K}{K+1}\right)\frac{r}{(1-r)^{2(1-\alpha)}}-\frac{1}{2^{2(1-\alpha)}}=0.
	\end{align*}
	The number $R^*_{k,\alpha}$ is sharp if $R^*_{k,\alpha}\leq 1/3$. 
	
\end{cor}
For $\alpha$ in the interval $[0, \frac{1}{2}\log_3(3+3k)]$, the Bohr radius $R_{k,\alpha}$ remains less than or equal to $1/3$, whereas for $\alpha$ in $(\frac{1}{2}\log_3(3+3k), 1)$, it becomes strictly greater than $1/3$
\begin{rem}
	If we set $K=1$ and $f=h$, then Corollary \ref{Cor-3.2} reduces to the results in \cite[Corollary 3.4]{Hamada-JMAA-2021} and \cite[Corollary 1]{Anand-Jain-Kumar-BMMSS-2021}
\end{rem}
\begin{rem}
	The result of Corollary \ref{Cor-3.2} coincides with that of \cite[Theorem 3]{Bhowmik-Das-JMAA-2018} for $K=1$ and $\alpha \in [0, 1/2]$.
\end{rem}
\begin{rem}
	If we take $\phi(z)=\frac{1+z}{1-z}$, the class $\mathcal{S}^*(\phi)$ reduces to the class $\mathcal{S}^*$ of starlike functions. For this choice of $\phi$, it can be easily seen  that 
	\begin{align*}
		h_{\phi}(z)=\frac{z}{(1-z)^2}\;\;\; \mbox{and}\;\;\; \hat{h}_{\phi}(r)=\frac{r}{(1-r)^2}.
	\end{align*}
\end{rem}
With this choice $\phi(z) = \frac{1+z}{1-z}$, Theorem \ref{lem-3.2} directly leads to the result shown in the following corollary.
\begin{cor}\label{Cor-3.3}
	Suppose that $F(z)=f(z)+\overline{g(z)}=\sum_{n=1}^{\infty}a_nz^n+\overline{\sum_{n=1}^{\infty}b_nz^n}$ is a sense-preserving $K$-quasiconformal harmonic mapping in $\mathbb{D}$ and $f\prec h$, where $h\in\mathcal{S}^*$. Then 
	\begin{align*}
		\sum_{n=1}^{\infty}(|a_n|+|b_n|)r^n\leq d(h(0),\partial h(\mathbb{D}))
	\end{align*}
	holds for $|z|=r\leq R_k= (5K+1-2\sqrt{6K^2+2K})/(K+1)$. The number $R_k$ is sharp.
\end{cor}
\begin{rem}
	If we consider $K=1$ (which corresponds to $k=0$) and $f=h$ in Corollary \ref{Cor-3.3}, then $R_0=3-2\sqrt{2}$ becomes the sharp Bohr radius for the class of starlike functions. 
\end{rem}
\begin{rem}
	If we take $K=1$ (which corresponds to $k=0$), we see that $r^*_{0,\phi}$ is the root in $(0,1)$ of the equation $Q_{k,\phi}(r):=\hat{h}_\phi(r)+h_\phi(-1)=0$. This indicates that the Theorem \ref{lem-3.2} includes the result of      \cite[Theorem 3.1]{Hamada-JMAA-2021} as a special case.
\end{rem}
\begin{proof}[\bf Proof of Theorem \ref{thm-3.2}]
	Let $h\in \mathcal{S}^*(\phi)$ with the Taylor expansion $h(z)=z+\sum_{n=2}^{\infty}e_nz^n$. Since $f\prec h$, by applying Lemma \ref{lem-3.1}, we have
	\begin{align}\label{eq-3.8}
		\sum_{n=1}^{\infty}|a_n|r^n\leq\sum_{n=1}^{\infty}|e_n|r^n\;\;\mbox{for}\;\; |z|=r\leq 1/3.
	\end{align}
	Since $h(z)/z\prec h_{\phi}(z)/z$, it follows along with the inequality \eqref{eq-3.8} that 
	\begin{align}\label{eq-3.9}
		\sum_{n=1}^{\infty}|a_n|r^n\leq\sum_{n=1}^{\infty}|e_n|r^n\leq\sum_{n=1}^{\infty}|d_n|r^n\;\;\mbox{for}\;\; |z|=r\leq 1/3.
	\end{align}
	By applying the similar arguments as in the prof of Theorem \ref{thm-2.1} and Lemma \ref{lem-2.2}, we obtain
	
	\begin{align}\label{eq-3.10}
		\sum_{n=1}^{\infty}|b_n|r^n\leq k\sum_{n=1}^{\infty}|a_n|r^n\leq k\sum_{n=1}^{\infty}|d_n|r^n \;\;\mbox{for}\;\; |z|=r\leq 1/3.
	\end{align}
	By using \eqref{eq-3.9} and \eqref{eq-3.10}, we have 
	\begin{align}\label{eq-3.11}
		C_2(r):=\sum_{n=1}^{\infty}(|a_n|+|b_n|)r^n\leq& (k+1)\sum_{n=1}^{\infty}|d_n|r^n\\\nonumber=&(k+1)\hat{h}_{\phi}(r)\\=&Q_{k,\phi}(r)-h_{\phi}(-1)\nonumber,
	\end{align}
	where $Q_{k,\phi}(r)=(k+1)\hat{h}_p(r)+h_{\phi}(-1)$. It is obvious that $Q_{k,\phi}$ is continuous and strictly increasing function on $(0,1)$. Also, we can deduce the following inequality by letting $r\to 1$ in the inequality of Lemma \ref{lem-3.3} that
	\begin{align*}
		-h_{\phi}(-1)\leq h_{\phi}(1)\leq\hat{h}_{\phi}(1)\leq(k+1)\hat{h}_{\phi}(1).
	\end{align*}
	It follows that $(k+1)\hat{h}_{\phi}(1)+h_{\phi}(-1)\geq 0$ and hence $\lim_{r\to 1}Q_{k,\phi}(r)=(k+1)\hat{h}_{\phi}(1)+h_{\phi}(-1)\geq 0$.\vspace{2mm}
	
	\noindent{\bf Case I:} $\lim_{r\to 1}Q_{k,\phi}(r)=0$.\vspace{1.2mm}
	
	Since $Q_{k,\phi}(r)$ is strictly increasing function on $(0,1)$, we obtain $Q_{k,\phi}(r)<0$ for all $r\in(0,1)$. Since for all $z\in\mathbb{D}$, $|h(z)|\geq -h_{\phi}(-r)$, we obtain $ d(h(0),\partial h(\mathbb{D}))\geq -h_{\phi}(-1)$. Then from the inequality \eqref{eq-3.11}, we have $C_2(r)\leq-h_{\phi}(-1)\leq d(h(0),\partial h(\mathbb{D}))$ for $|z|=r\leq 1/3$. \vspace{1.2mm} 
	
	\noindent{\bf Case II:} $\lim_{r\to 1}Q_{k,\phi}(r)>0$.\vspace{1.2mm}
	
	It is noted that $Q_{k,\phi}(0)=h_{\phi}(-1)<0$ ensures the existence of a unique root of $Q_{k,\phi}(r)=0$ in $(0,1)$, denoted by $r^*_{k,\phi}$. Then we have $Q_{k,\phi}(r)\leq0$ for $|z|=r\leq r^*_{k,\phi}$. This gives 
	\begin{align*}
		C_2(r)\leq-h_{\phi}(-1)\leq d(h(0),\partial h(\mathbb{D}))\;\; \mbox{for}\;\; |z|=r\leq \min\{1/3,r^*_{k,\phi}\}.
	\end{align*}
	To establish the sharpness of the radius obtained in {Case II} when $r^*_{k,\phi}\leq 1/3$, we consider the function $F(z)=f(z)+\overline{g(z)}$ in $\mathbb{D}$, where $f=h=h_{\phi}$ and $g=k\lambda h_{\phi}$ with non-negative Taylor coefficient of $h_{\phi}$, and $|\lambda|=1$. Then the facts $\hat{h}_{\phi}=h_{\phi}$ and $d(h_{\phi}(0),\partial h_{\phi}(\mathbb{D})=-h_{\phi}(-1)$ give that
	\begin{align*}
		\sum_{n=1}^{\infty}(|d_n|+|k\lambda d_n|)r^n=(k+1)\hat{h}_{\phi}(r)>-h_{\phi}(-1)=d(h_{\phi}(0),\partial h_{\phi}(\mathbb{D})
	\end{align*}
	holds for $r>r^*_{k,\phi}$. This shows that the radius $r^*_{k,\phi}$ cannot be improved. 
\end{proof}

\noindent{\bf Acknowledgment:} The first author is supported by Science and Engineering Research Board (SERB) (File No. SUR/2022/002244), Govt. of India, and the second author is supported by UGC-JRF (NTA Ref. No.: $ 221610103011$), New Delhi, India. \vspace{2mm}

\noindent\textbf{Compliance of Ethical Standards}\\

\noindent\textbf{Conflict of interest.} The authors declare that there is no conflict  of interest regarding the publication of this paper.\vspace{1.5mm}

\noindent\textbf{Data availability statement.}  Data sharing not applicable to this article as no datasets were generated or analyzed during the current study.\vspace{1.5mm}

\noindent\textbf{Funding.} No fund.


\begin{thebibliography}{200}
	
\bibitem{Abu-Muhanna-CVEE-2010} {\sc Y. Abu-Muhanna}, Bohr’s phenomenon in subordination and bounded harmonic classes, {\it Complex Var. Elliptic Equ.} {\bf 55} (2010), 1071–1078.

\bibitem{Abu-Muhanna-Ali-JMAA-2011} {\sc Y. Abu-Muhanna} and  {\sc R. M. Ali}, Bohr’s phenomenon for analytic functions into the exterior of a compact convex body, {\it J. Math. Anal. Appl.} {\bf 379} (2011), 512–517.

\bibitem{Abu-Muhanna-Ali-Ponnusami-2016} {\sc Y. Abu-Muhanna}, {\sc R. M. Ali}, and {\sc S. Ponnusamy}, On the Bohr inequality, in: N.K. Govil, et al. (Eds.), {\it  Progress in Approximation Theory and Applicable Complex Analysis}, in: Springer Optimization and Its Applications, vol. 117, 2016, pp. 265–295.

\bibitem{Ahamed-Ahammed-MJM-2024} {\sc M. B. Ahamed} and {\sc S. Ahammed}, Bohr Inequalities for Certain Classes of Harmonic Mappings, {\it Mediterr. J. Math.} {\bf 21}(1) (2024), 21.

\bibitem{Ahamed-Allu-Halder-AFM-2022} {\sc M. B. Ahamed}, {\sc V. Allu}, and {\sc H. Halder}, The Bohr phenomenon for analytic functions on a
shifted disk, {\it Ann. Fenn. Math.} {\bf 47} (2022), 103–120.

\bibitem{Ali-Jain-Ravichandran-RM-2019} {\sc R. M. Ali}, {\sc N. K. Jain}, and {\sc V. Ravichandran}, Bohr radius for classes of analytic
functions, {\it Results Math.} {\bf 74}(4) (2019), 179.

\bibitem{Alkhaleefah-Kayumov-Ponnusamy-PAMS-2019} {\sc S. A. Alkhaleefah}, {\sc I. R. Kayumov}, and {\sc S. Ponnusamy}, On the Bohr inequality with a fixed zero coefficient, {\it Proc. Amer. Math. Soc.} {\bf 147} (2019), 5263–5274.


\bibitem{Allu-Halder-JMAA-2021} {\sc V. Allu} and {\sc H. Halder}, Bohr radius for certain classes of starlike and convex univalent functions,
{\it J. Math. Anal. Appl.} {\bf 493}(1) (2021), 124519.

\bibitem{Anand-Jain-Kumar-BMMSS-2021} {\sc S. Anand}, {\sc N. K. Jain}, and {\sc S. Kumar}, Sharp Bohr radius constants for certain analytic functions, {\it Bull. Malays. Math. Sci. Soc.} {\bf 44} (2021), 1771-1785.

\bibitem{Arora-Vinayak-CVVE-2025} {\sc V. Arora} and {\sc M. Vinayak}, Bohr's phenomenon for certain classes of analytic functions, {\it Complex Var. Elliptic Equ.}, 1–23. https://doi.org/10.1080/17476933.2024.2439960

\bibitem{Avkhadiev-Pommerenke-Wirths-MN-2004} {\sc F. G. Avkhadiev}, {\sc Ch. Pommerenke}, and {\sc K. -J. Wirths}, On the coefficients of concave univalent functions, {\it Math. Nachr.} {\bf 271} (2004), 3–9.

\bibitem{Avkhadiev-Wirths-FM-2007} {\sc F. G. Avkhadiev} and {\sc  K. -J. Wirths}, A proof of Livingston conjecture, {\it Forum Math.} {\bf 19} (2007), 149–158.

\bibitem{Avkhadiev-Wirths-CVEE-2007} {\sc F. G. Avkhadiev} and {\sc  K. -J. Wirths}, Subordination under concave univalent functions, {\it Complex Var. Elliptic Equ.} {\bf 52}(4) (2007), 299–305.

\bibitem{Avkhadiev-Wirths-2009} {\sc F. G. Avkhadiev} and {\sc K. -J. Wirths}, Schwarz-Pick Type Inequalities, Frontiers in Mathematics, Birkh$\ddot{a}$user Verlag, Basel, 2009, viii+156 pp.



\bibitem{Bhowmik-Das-JMAA-2018} {\sc B. Bhowmik} and {\sc N. Das}, Bohr phenomenon for subordinating families of certain univalent
functions, {\it J. Math. Anal. Appl.} {\bf  462}(2) (2018), 1087–1098.

\bibitem{Bhowmik-Ponnusamy-Wirths-KMJ-2007} {\sc B. Bhowmik}, {\sc S. Ponnusamy}, and {\sc K. -J. Wirths}, Domains of variability of Laurent coefficients and the convex hull for the family of concave univalent functions, {\it Kodai Math. J.} {\bf 30} (2007), 385–393.

\bibitem{Bohr-PLMS-1914} {\sc H. Bohr}, A theorem concerning power series, {\it Proc. Lond. Math. Soc.} {\bf s2-13} (1914), 1–5.

	

\bibitem{Djakov-JA-2000} {\sc P. B. Djakov} and {\sc M. S. Ramanujan}, A remark on Bohr’s theorem and its generalizations, \textit{J. Anal.} \textbf{8}(2000), 65--77.


\bibitem{Duren-1983} {\sc P. L. Duren}, Univalent functions, Springer, New York, 1983. 

\bibitem{Duren-2004} {\sc P. L. Duren}, Harmonic mapping in the plane, Cambridge University Press (2004).

 \bibitem{Evdoridis-Ponnusamy-Rasila} {\sc S. Evdoridis}, {\sc S. Ponnusamy}, and {\sc A. Rasila}, Improved Bohr’s inequality for shifted disks, {\it  Results Math.} {\bf 76} (14) (2021).
	
\bibitem{Fournier-Ruscheweyh-CRM-2010} {\sc R. Fournier} and {\sc St. Ruscheweyh}, On the Bohr radius for a simply connected plane domain,
Centre de Recherches Mathematiques CRM Proceeding and Lecture Notes, {\bf 51} (2010), 165–171.

\bibitem{Gangania-Kumar-MJM-2022} {\sc K. Gangania} and {\sc S. S. Kumar}, Bohr-Rogosinski phenomenon for $\mathcal{S}^*(\psi)$ and $\mathcal{C}(\psi)$, {\it Mediterr. J. Math.} {\bf 19} (2022), 161.

\bibitem{Garcia-Mashreghi-Ross-2018} {\sc S. R. Garcia}, {\sc J. Mashreghi}, and {\sc W. T. Ross}, Finite Blaschke Products and Their Connections, Springer, Cham, 2018.


\bibitem{Hamada-JMAA-2021} {\sc H. Hamada}, Bohr phenomenon for analytic functions subordinate to starlike or convex functions,
{\it J. Math. Anal. Appl.} {\bf 499}(1) (2021),125019.

\bibitem{Hamada-Honda-Kohr-AAMP-2025} {\sc H. Hamada, T. Honda,} and {\sc M. Kohr,} Bohr–Rogosinski radius for holomorphic mappings with values in higher dimensional complex Banach spaces, \textit{Anal. Math. Phys.} (2025) 15:64, https://doi.org/10.1007/s13324-025-01061-x.

\bibitem{Ismagilov-Kayumov-Ponnusamy-JMAA-2020} {\sc A. A. Ismagilov}, {\sc I. R. Kayumov}, and {\sc S. Ponnusamy}, Sharp Bohr type inequality, {\it J. Math. Anal. Appl.} {\bf 489} (2020) 124147, 10 pp.


\bibitem{Janowski-APM-1973} {\sc W. Janowski}, Some extremal problems for certain families of analytic functions I, {\it Ann. Polon. Math.} {\bf 28} (1973),297–326. doi: 10.4064/ap-28-3-297-326

\bibitem{Janowski-BAPSSSMAP-1973} {\sc W. Janowski}, Some extremal problems for certain families of analytic functions II, {\it Bull. Acad. Polon. Sci. Ser. Sci. Math. Astron. Phys.} {\bf 21} (1973),17–25.

\bibitem{Jenkins-MMJ-1962} {\sc J. A. Jenkins}, On a conjecture of Goodman concerning meromorphic univalent functions, {\it Mich. Math. J.} {\bf 9} (1962) 25–27.

\bibitem{Kalaj-MZ-2008} {\sc D. Kalaj}, Quasiconformal harmonic mapping between Jordan domains, {\it Math. Z.} {\bf 260}(2) (2008), 237–252.

\bibitem{Kaplan-MMJ-1952} {\sc W. Kaplan}, Close-to-convex schlicht functions, {\it Mich. Math. J.} {\bf 1} (1952), 169–185. 

\bibitem{Kayumov-Ponnusamy-CMFT-2017} {\sc I. R. Kayumov} and {\sc S. Ponnusamy}, Bohr inequality for odd analytic functions, {\it Comput. Methods	Funct. Theory} {\bf 17} (2017), 679–688.



\bibitem{Kayumov-Ponnusamy-AASFM-2019} {\sc I. R. Kayumov} and {\sc S. Ponnusamy}, On a powered Bohr inequality, {\it Ann. Acad. Sci. Fenn. Math.} {\bf 44} (2019), 301–310.

\bibitem{Kayumov-Ponnusamy} {\sc I. R. Kayumov} and {\sc S. Ponnusamy}, Bohr-Rogosinski radius for analytic functions, preprint,
https://doi.org/10.48550/arXiv.1708.05585.



\bibitem{Kayumov-Ponnusamy-Shakirov-MN-2018}  {\sc I. R. Kayumov}, {\sc S. Ponnusamy}, and {\sc N Shakirov}, Bohr radius for locally univalent harmonic mappings, {\it Math. Nachr.} {\bf 291} (11-12) (2018), 1757-1768.


\bibitem{Lewy-BAMS-1936} {\sc H. Lewy}, On the non-vanishing of the Jacobian in certain one-to-one mappings, {\it Bull. Am. Math. Soc.} {\bf 42} (1936), 689–692.

\bibitem{Liu-Ponnusamy-BMMSS-2019} {\sc Z. H. Liu} and {\sc S. Ponnusamy}, Bohr radius for subordination and $K$-quasiconformal harmonic
mappings, {\it Bull. Malays. Math. Sci. Soc.} {\bf 42} (2019), 2151–2168.



\bibitem{Liu-Ponnusami-Wang-2020} {\sc M. S. Liu}, {\sc S. Ponnusamy}, and {\sc J. Wang}, Bohr’s phenomenon for the classes of Quasisubordination
and K-quasiregular harmonic mappings, {\it Rev. Real Acad. Cienc. Exactas Fis. Nat.- A: Mat.} {\bf 114} (2020), 115.

\bibitem{Ma-Minda-1992} {\sc W. C. Ma} and {\sc D. Minda}, A unified treatment of some special classes of univalent functions. In: Proceedings
of the Conference on Complex Analysis (Tianjin, 1992), Conf. Proc. Lecture Notes Anal. I, Int. Press, Cambridge, p. 157–169.

\bibitem{Martio-AASFAI-1968} {\sc O. Martio}, On harmonic quasiconformal mappings, {\it Ann. Acad. Sci. Fenn. A. I.} {\bf 425} (1968), 3–10.


\bibitem{Paulsen-Popescu-Singh-PLMS-2002} {\sc V. I. Paulsen}, {\sc G. Popescu}, and {\sc D. Singh}, On Bohr's inequality, {\it Proc. Lond. Math. Soc.} {\bf 85} (2) (2002), 493-512.

\bibitem{Paulsen-Singh-2022} {\sc V. I. Paulsen} and {\sc D. Singh}, A simple proof of Bohr’s inequality, available from $arXiv:2201.10251v1$.

\bibitem{Ponnusamy-Wirths-CMFT-2020} {\sc S. Ponnusamy} and {K. -J. Wirths}, Bohr type inequalities for for functions with a multiple zero at the origin, {\it Comput. Methods Funct. Theory} {\bf 20} (2020), 559-570.

\bibitem{Robertson-AM-1936} {\sc M. S. Robertson}, On the theory of univalent functions, {\it Ann. Math.} {\bf 37} (1936),374–408. doi:
10.2307/1968451

\bibitem{Rogosinski-MJ-1923} {\sc  W. Rogosinski}, Uber Bildschranken bei Potenzreihen und ihren Abschnitten, {\it Math. Z.} {\bf 17} (1923), 260-276.

\bibitem{Wirths-SMJ-2003} {\sc K. -J. Wirths}, The Koebe domain for concave univalent functions, {\it Serdica Math. J.} {\bf 29} (2003), 355–360

\bibitem{Wirths-SMJ-2006} {\sc K. -J. Wirths}, On the residuum of concave univalent functions, {\it  Serdica Math. J.} {\bf 32} (2006), 209–214.




\end{thebibliography}
\end{document}